\documentclass[12pt]{amsart}
\usepackage{amssymb}
\usepackage{amsmath}
\usepackage{amsfonts}

\usepackage[letterpaper, portrait, margin=0.9in]{geometry}
\usepackage{hyperref,color,mathrsfs,extarrows}
\usepackage{cleveref}
\theoremstyle{plain}
\newtheorem{theorem}{Theorem}[section]

\newtheorem{lemma}[theorem]{Lemma}

\theoremstyle{definition}

\newtheorem{remark}[theorem]{Remark}

\newcommand{\LL}{\mathsf{LL}}

\title[Dimension jump at $p_u$ for percolation in $\infty+d$ dimensions]{Dimension jump at the uniqueness threshold for percolation in $\infty+d$ dimensions}

\author{Tom Hutchcroft and Minghao Pan}

\begin{document}

\maketitle

\begin{abstract}
Consider percolation on $T\times \mathbb{Z}^d$, the product of a regular tree of degree $k\geq 3$ with the hypercubic lattice $\mathbb{Z}^d$. It is known that this graph has $0<p_c<p_u<1$, so that there are non-trivial regimes in which percolation has $0$, $\infty$, and $1$ infinite clusters a.s., and it was proven by Schonmann (1999) that there are infinitely many infinite clusters a.s.\ at the uniqueness threshold $p=p_u$. We strengthen this result by showing that 
the Hausdorff dimension of the set of accumulation points of each infinite cluster in the boundary of the tree 
   has a jump discontinuity from at most $1/2$ to $1$ at the uniqueness threshold~$p_u$. We also prove that various other critical thresholds including the $L^2$ boundedness threshold $p_{2\to 2}$ coincide with $p_u$ for such products, which are the first nonamenable examples proven to have this property. All our results apply more generally to products of trees with arbitrary infinite amenable Cayley graphs and to the lamplighter on the tree.
\end{abstract}

\section{Introduction}
Let $G=(V,E)$ be a connected, locally finite, infinite graph. We assume that $G$ is transitive, meaning that any vertex of the graph can be mapped to any other vertex by a graph automorphism, and will write $o$ for a fixed root vertex of $G$. We denote \textbf{Bernoulli bond percolation} with parameter $p$ on $G$ by $G_p$, which is obtained from $G$ by independently setting each edge to be open (retained) with probability $p$ and closed (removed) with probability $1-p$. We use $\mathbb{P}_p$ and $\mathbb{E}_p$ to denote probabilities and expectations taken with respect to the random graph $G_p$. Connected components of $G_p$ are called \textbf{clusters} and, for each $x\in V(G)$, the cluster containing $x$ is denoted by $K_x$. We define the critical parameter to be
\begin{equation*}
p_{c}=p_{c}(G):=\inf \left\{ p\in \lbrack 0,1]:
\text{there exists an infinite cluster $\mathbb{P}_{p}$ a.s.}\right\} 
\end{equation*}
and the \textbf{uniqueness threshold} to be%
\begin{equation*}
p_{u}=p_{u}(G)=\inf \left\{ p\in \lbrack 0,1]:
\text{there exists a unique infinite cluster 
$\mathbb{P}_{p}$ a.s.}\right\};
\end{equation*}
for transitive graphs these are the only values of $p$ at which the number of infinite clusters changes by the results of \cite{HPS99,MR1676831}.
The classic Burton-Keane argument \cite{burton1989density} establishes that $p_c=p_u$ for amenable transitive graphs; the converse statement that $p_c<p_u$ for every nonamenable transitive graph is a well-known conjecture of Benjamini and Schramm \cite{bperc96} which has been proven in a number of cases \cite{bperc96,1804.10191,Hutchcroftnonunimodularperc,MR3009109,MR3005730,MR1756965,MR1833805}
but remains open in general.

An important heuristic in the study of \emph{critical} ($p=p_c$) percolation is that in sufficiently ``high-dimensional'' contexts, critical percolation clusters should, on large scales, look like critical \emph{branching random walk}, which can be thought of as a ``non-interacting'' variant of percolation. 
In particular, in such high-dimensional settings critical percolation should be governed by the same critical exponents as percolation on a \emph{tree}, so that e.g.\ the probability that the cluster of the origin has size at least $n$ scales like $n^{-1/2}$.
This is known as \textbf{mean-field} critical behaviour: it is expected to hold on transitive graphs if and only if the volume growth is strictly more than six-dimensional, and is known to hold on high-dimensional Euclidean lattices \cite{MR762034,MR1127713,duminil2024alternative,fitzner2015nearest,MR1043524} and various classes of nonamenable transitive graphs \cite{1804.10191,Hutchcroftnonunimodularperc,MR1888869}.

This paper is motivated by the problem of understanding when this (imprecisely defined) comparison between percolation and branching random walk can be extended
all the way to the uniqueness threshold $p_u$. 
The appropriate analogue of $p_u$ for branching random walk is the \textbf{recurrence threshold}, where the mean offspring $\bar \mu$ is equal to $1/\rho$ with $\rho$ the spectral radius of the graph: This is indeed a good analogue of $p_u$ in the sense that if $\bar \mu <1/\rho$ then the branching random walk visits a zero density set of vertices a.s.\ while if $\bar \mu>1/\rho$ then the branching random walk visits every vertex infinitely often a.s.\ \cite{gantert2005critical}. In fact, branching random walk at the recurrence threshold $\bar \mu=1/\rho$ itself also always belongs to the transient regime in the sense that it visits a zero-density set of vertices a.s.\ \cite[Theorem 3.2]{gantert2005critical} (see also \cite[Theorem 7.8]{Woess}); see also \cite{hutchcroft2020non} for stronger results in this direction. This can be thought of as the branching random walk analogue of having non-unique infinite clusters at $p_u$ for percolation.

In light of this, an immediate problem that arises in our investigation is that the tree is no longer a good model for what should happen in a ``generic'' or ``high-dimensional'' example at the uniqueness threshold $p_u$, since it has $p_u=1$ and therefore has a \emph{unique} infinite cluster at $p_u$. Thus, the tree's critical behaviour at $p_u$ does not even \emph{qualitatively} coincide with the critical behaviour of branching random walk at the recurrence threshold. Another important class of examples that are non-generic in this sense are the  tesselations of the hyperbolic plane, which have $0<p_c<p_u<1$ but have a unique infinite cluster at $p_u$ in contrast to branching random walk~\cite{BS00,MR3009109}. This is conjecturally related to the fact that, although these examples are nonamenable and therefore ``infinite-dimensional'' in various senses (e.g.\ they have infinite volume-growth, spectral, and isoperimetric dimensions)
 they are ``low-dimensional'' in various other senses that are perhaps more relevant to percolation at the uniqueness threshold (e.g.\ they have internal isoperimetric dimension \cite{hutchcroft2021non} equal to $1$ and positive first $\ell^2$-Betti number \cite[Chapter 7.8]{LP:book}). These properties distinguish these examples from various other classes of nonamenable transitive graphs that are known to have \emph{non-uniqueness} at $p_u$, including nonamenable products of infinite graphs \cite{gaboriau2016approximations,MR1770624}, Cayley graphs of Kazhdan groups \cite{LS99},  and Cayley graphs of nonamenable groups with amenable normal subgroups of exponential growth such as nonamenable lamplighter groups $\mathbb{Z}_2 \wr\Gamma$ \cite{2409.12283}. 


The primary goal of this paper is to show that the analogy between percolation and branching random walk can be extended all the way to $p_u$ in various strong, quantitative senses for the product $T\times \mathbb{Z}^d$ of a $k$-regular tree and a lattice with $k\geq 3$. This example has a long history in percolation theory: It was first studied by Grimmett and Newman in 1990 \cite{MR1064560}, who showed that $0<p_c<p_u<1$ when $k$ is sufficiently large, making it the first example proven to have this property. It was also the first example proven to have infinitely many infinite clusters at $p_u$, as shown by Schonmann in 1999 \cite{schonmann1999percolation}. Finally, it was a central motivating example in the development of the first author's theory of percolation on \emph{nonunimodular} transitive graphs \cite{Hutchcroftnonunimodularperc}, which in particular yields an extension of the results of \cite{MR1064560} to all $k\geq 3$. Our results also apply to the lamplighter graph $\LL(T)$ over the tree, where they are easier to prove.

\medskip

\textbf{The dimension jump.} One of the most interesting features of the branching random walk recurrence transition is its highly discontinuous character, exemplified by the \emph{dimension jump} on the boundary: 
It was proven by Liggett \cite{Liggett96} and Heuter and Lalley \cite{HL00} that for branching random walk on the free group, the Hausdorff dimension of the set of limit points of the process on the boundary of the tree jumps discontinuously from at most $1/2$ at and below the recurrence threshold to $1$ above the recurrence threshold. Similar results have now been proven for branching random walks in a wide range of geometric settings \cite{CGM12,dussaule2022branching,karpelevich1998phase,lalley1997hyperbolic,SWX23} as well as for the contact process \cite{liggett1996branching}. Similar results were also established by Liggett \cite{liggett1996branching} for the \emph{contact process} on a tree, which can (via its graphical representation) be thought of as a percolation-type model on $T\times \mathbb{R}$ that is oriented in the $\mathbb{R}$ coordinate. See also Lalley's ICM proceedings \cite{lalley2006weak} for a survey of the subject as of 2006.

Again, these discontinuous behaviours of branching random walk at the recurrence threshold are in stark contrast to the behaviour of percolation on trees and the hyperbolic plane. Indeed, for percolation on the tree it follows from Hawkes' theorem \cite{Hawkes81} that the limit sets of infinite clusters have Hausdorff dimension $\log p(k-1)/\log (k-1)$ when $p_c=1/(k-1)<p\leq 1$, which increases continuously from $0$ at $p_c$ to $1$ at $p_u=1$. (Here and elsewhere we use the metric $d_{\partial T}(\xi,\zeta)=(k-1)^{-|\xi\wedge \zeta|}$ on $\partial T$, where $\xi \wedge \zeta$ denotes the vertex of $T$ where the rays to the boundary points $\xi$ and $\zeta$ last meet and $|\xi\wedge \zeta|$ denotes the distance from this point to the origin, which makes the full boundary $\partial T$ have Hausdorff dimension $1$.) Lalley~\cite{Lalley01} proved that a similar qualitative picture applies for percolation on certain tessellations of the hyperbolic plane, with the Hausdorff dimension of the limit set of an infinite cluster growing continuously from $0$ to $1$ as $p$ increases from $p_c$ to $p_u$.

 Our first main result shows that percolation on $T\times \mathbb{Z}^d$ and the lamplighter graph $\LL(T)$ has a dimension jump at $p_u$ analogous to that for branching random walk. These are the first examples proven to exhibit such a dimension jump for unoriented percolation.

\begin{theorem}[Dimension jump]
\label{thm:dimension_jump}
Let $T$ be the $k$-regular tree for some $k\geq 3$ and consider percolation on either the direct product $T\times H$ of $T$ with some infinite amenable Cayley graph $H$ or the lamplighter graph $\LL(T)$. Let $\pi$ denote the projection onto the tree, and let $\Lambda$ be the set of limit points in $\partial T$ of the projected cluster $\pi(K_o)$. On the event that $K_o$ is infinite, the Hausdorff dimension $\delta_H(\Lambda)$ coincides $\mathbb{P}_p$-almost surely with a constant $\delta_H(p)$ which increases continuously from $0$ to $\delta_H(p_u)\leq \frac{1}{2}$ over the interval $(p_c,p_u]$ and is equal to $1$ for all $p>p_u$.
\end{theorem}


In the case of $T\times \mathbb{Z}$, \Cref{thm:dimension_jump} can be thought of as a significant strengthening of a result of Schonmann \cite{schonmann1999percolation}, who proved that this graph has non-uniqueness at $p_u$ by establishing the two-point function estimate
\[
\mathbb{P}_{p_u}((x,0)\leftrightarrow (y,0))\leq (k-1)^{-\frac{1}{2}d(x,y)}
\]
for every $x,y \in V(T)$; the factor $1/2$ appearing in the exponent here is closely related to the inequality $\delta_H(p_u)\leq 1/2$. Probabilistically, this factor arises from the birthday problem: it is the value of the constant in the exponent at which the expected intersection of \emph{two independent clusters} and the level set $\{(x,0):x\in V(T)\}$ becomes infinite.
To prove \Cref{thm:dimension_jump} in this case, the key step is to improve Schonmann's estimate to a point-to-fiber estimate of the form
\begin{equation}
\label{eq:point_to_fiber_intro}
\mathbb{P}_{p}\bigl((x,0)\leftrightarrow \{(y,n):n\in \mathbb{Z}\}\bigr)\leq C_p (k-1)^{-\frac{1}{2}d(x,y)} 
\end{equation}
for every $p<p_u$ and $x,y\in V(T)$. Our proof of this stronger estimate relies in part on the ``subgroup relativization'' techniques developed in our companion paper \cite{2409.12283}, which guarantee that the total size of the intersection of the cluster of $x$ with the fiber $\{(y,n):n\in \mathbb{Z}\}$ has an exponential tail whenever $p<p_u$, uniformly in $x$ and $y$.

\medskip

\textbf{The $L^2$ boundedness threshold.}
Our next result concerns the $L^2$ boundedness threshold $p_{2\to 2}$ as introduced in \cite{1804.10191}. 
 Given a non-negative matrix $M$, we define the $L^2\to L^2$ operator norm of $M$ by
\[\|M\|_{2\rightarrow 2}:=\sup\{\|Mf\|_{2}: f \text{ finitely supported, }\|f\|_{2}=1\}.\]
Let $G$ be a connected, locally finite transitive graph, let $\tau_p(x,y)$ be the \textbf{two-point function} (i.e., the probability that $x$ is connected to $y$ in $G_p$), and let $T_p\in [0,\infty]^{V^2}$ denote the matrix $T_p(x,y):=\tau_p(x,y)$. We define the critical value
\[p_{2\rightarrow 2}:=\sup\{p\in [0,1]:\|T_p\|_{2\rightarrow 2}<\infty \}.\]
It is easily seen that $p_c\leq p_{2\rightarrow 2}\leq p_u$ for any transitive graph. 
 In \cite{1804.10191}, it is proven that $p_c<p_{2\rightarrow 2}$ for every transitive, nonamenable, Gromov hyperbolic graph and conjectured that the same strict inequality holds for all connected, locally finite, transitive, nonamenable graphs; this conjecture, if true, would imply that $p_c<p_u$ as well as many other strong quantitative statements about critical and near-critical percolation \cite{hutchcroft20192,hutchcroft2022slightly,hutchcroft2024slightly}. Moreover, it is shown in \cite{hutchcroft2024slightly} that infinite clusters of percolation in the regime $p_c<p<p_{2\to 2}$ have tree-like geometry in various strong quantitative senses, strengthening the analogy with branching random walk. It was also shown in \cite[Section 6.2]{1804.10191} that infinite clusters are always non-unique at $p_{2\to 2}$ on any transitive graph, so that e.g.\ tesselations of the hyperbolic plane must have the strict inequality $p_c<p_{2\to 2}<p_u$. While it is plausible that the equality $p_u=p_{2\to 2}$ holds in many of the examples that are known to have non-uniqueness at $p_u$, it remained open to demonstrate a single nonamenable transitive graph with this property.

\medskip

Our next main result establishes $T\times \mathbb{Z}^d$ and $\LL(T)$ as the first nonamenable examples proven to have $p_{2\to2}=p_u$. (Both examples are known to have $p_c<p_u$ by the results of \cite{Hutchcroftnonunimodularperc}.)

\begin{theorem}
\label{thm:lamplighter_on_tree_p22}
Let $T$ be a $k$-regular tree for some $k\geq 3$. If $G$ is either $\LL(T)$ or $T\times H$ for some infinite amenable Cayley graph $H$, then the critical values $p_{2\to 2}(G)$ and $p_u(G)$ coincide. 
\end{theorem}

As we shall see, this theorem is closely related to \Cref{thm:dimension_jump}. To get some indication of this, note that if $p<p_{2\to 2}$ then we can compute via Cauchy-Schwarz that
\[\mathbb{E}|K_o \cap F| = \langle \mathbf{1}_o, T_p \mathbf{1}_F\rangle \leq \|T_p\|_{2\to 2} \|\mathbf{1}_o\|_2^{1/2}\|\mathbf{1}_F\|_2^{1/2} \leq C_p |F|^{1/2}
\]
for every finite set $F \subseteq V$. Intuitively, the exponent $1/2$ appearing here should be thought of as ``the same'' $1/2$ that appears in the dimension jump established in \Cref{thm:dimension_jump}; both results can be thought of as saying that infinite clusters below $p_u$ take up at most the square root of the available volume. Note also that the branching random walk two-point function is given by $G(x,y)=\sum_{n=0}^\infty (\overline{\mu})^n P^n(x,y)$ and therefore defines a bounded operator on $L^2$ if and only if $\overline{\mu}<1/\rho$ is below the recurrence threshold. As such, \Cref{thm:lamplighter_on_tree_p22} can be thought of as another instance of percolation having similar behaviour to branching random walk up to the uniqueness threshold.

\subsection{Nonunimodularity}
The proofs of \Cref{thm:dimension_jump,thm:lamplighter_on_tree_p22} will follow \cite{Hutchcroftnonunimodularperc} by making use of the \emph{nonunimodular} structure of the graphs $T \times H$ and $\LL(T)$. Before explaining this further, let us begin by giving precise definitions of the graphs under consideration.
 Let $T$ be the $k$-regular tree with $k\geq 3$, which can be thought of as a Cayley graph of the $k$-fold free product $\mathbb{Z}_2 * \mathbb{Z}_2 * \cdots * \mathbb{Z}_2$ or of the free group on $k/2$ generators if $k$ is even. 
The \textbf{direct product} $G\times H$ or two graphs $G=(V(G),E(G))$ and $H=(V(H),E(H))$ is defined to be the graph with edge set $V(G)\times V(H)$ and edge set $\{\{(v,x),(v,y)\} : v \in V(G), \{x,y\} \in E(H)\} \cup \{\{(u,x),(v,x)\} : x \in V(H), \{u,v\} \in V(G)\}$. 
 We define the \textbf{lamplighter graph} over $T$, denoted by $\operatorname{LL}(T)$, to be the graph with vertex set $V(T) \times \{$finitely supported functions $V(T)\to \{0,1\}\}$ formed by linking two vertices by an edge if one can be reached from the other either by flipping the lamp at the current location of the lighter or moving the lighter one step in the tree. (That is, we work with the ``switch or walk'' Cayley graph of the lamplighter group.)  This is a Cayley graph of the wreath product $\mathbb{Z}_2 \wr (\mathbb{Z}_2 * \mathbb{Z}_2 * \cdots * \mathbb{Z}_2)$, of which the lamp configuration group $\mathscr{L}:=\bigoplus_{v\in T} \mathbb{Z}_2$ is an amenable normal subgroup.

\begin{remark} All our methods also apply to the switch-walk-switch Cayley graph of the wreath product $\mathbb{Z}_2 \wr (\mathbb{Z}_2 * \mathbb{Z}_2 * \cdots * \mathbb{Z}_2)$, as well as to ``tensor products'' $T\times H$ where edges correspond to simultaneous moves in $T$ and $H$. (One can also replace $\mathbb{Z}_2$ with any finitely generated amenable lamp group.)  Our proofs do \emph{not} apply to arbitrary Cayley graphs of these groups, as it relies heavily on the compatibility of our graphs with the radial symmetries of the tree.
\end{remark}



Let $T$ be a $k$-regular tree with $k\geq 3$ and let $G$ be either $T\times H$ with $H$ an amenable Cayley graph or $\LL(T)$. In each case, there is a natural projection map $\pi:G\to T$ which can be interpreted as taking a quotient by a certain amenable subgroup of a group that $G$ is a Cayley graph of. Moreover, this quotient has the property that each automorphism of the tree can be lifted to an automorphism of $G$ that maps fibers to fibers. (Several of the intermediate statements we prove in this paper en route to our main theorems hold for any graph admitting a projection to a tree of this form.)
Fix an orientation of the tree $T$ so that each vertex has exactly one oriented edge emanating from it; picking such an orientation is equivalent to picking an end of the tree $T$ and orienting each edge towards this end. This orientation induces an integer-valued \textbf{height function} $h(u,v)$ on $T$, where $h(u,v)$ is the the number of edges crossed in the forward direction minus the number of edges crossed in the backwards direction in a path from $u$ to $v$. Using the height function, one gets a decomposition of $T$ into \textbf{levels}, so that each vertex has one neighbour in the level above it and $(k-1)$ neighbours in the level below it. This orientation of $T$ also induces a partial orientation on $G$, where only tree edges are oriented, along with an associated height function on $G$ defined by $h(x,y)=h(\pi(x),\pi(y))$. We write $\Gamma$ for the group of automorphisms of $G$ that preserve this orientation. The group $\Gamma$ is nonunimodular, with modular function
\[
\Delta(x,y) = (k-1)^{h(x,y)}.
\]
This means that if $F$ is a function sending pairs of vertices in  $G$ to $[0,\infty]$ such that $F(\gamma x, \gamma y)=F(x,y)$ for every two vertices $x,y$ and every $\gamma\in \Gamma$ then
\begin{equation}
\label{eq:tiltedMTP}
\sum_x F(o,x) = \sum F(x,o)\Delta(o,x).
\end{equation}
The identity \eqref{eq:tiltedMTP} is known as the \textbf{tilted mass-transport principle}. (For \emph{unimodular} transitive automorphism groups, the same identity holds with $\Delta\equiv 1$, in which case it is simply called the mass-transport principle.) Further background on nonunimodular transitive groups of automorphisms and their modular functions can be found in \cite{Hutchcroftnonunimodularperc}.

The nonunimodular transitive group of automorphisms $\Gamma$ allows us to define several further percolation thresholds.
For a vertex $x$ of $G$ and let $[x]:=\pi^{-1}(\pi(x))$ denote the set of vertices in $G$ that have the same projection to the tree as $x$. For each vertex $x$ of $G$ and $-\infty \leq n \leq m \leq \infty$, we write $L_{m,n}(x)$ for the set of vertices of $G$ whose height difference from $x$ is between $n$ and $m$, writing $L_n=L_{n,n}$. Given also a percolation configuration on $G$, we write $X^{n,m}_l(x)$ for the number of vertices in $L_l(x)$ that are connected to $x$ by an open path using only vertices of $L_{n,m}(x)$. In all cases, we will omit the argument $x$ when taking $x=o$.
The tilted mass-transport principle implies that
\[
\mathbb{E}_p X_l^{a,b} = (k-1)^{-l}\mathbb{E}_p X_{-l}^{a-l,b-l}
\]
for every $l\in \mathbb{Z}$ and $-\infty \leq a \leq b\leq + \infty$ with $a\leq l \leq b$.
 The \textbf{tilted susceptibility} is defined for each $p\in [0,1]$ and $\lambda \in \mathbb R$ by
\[
\chi_{p,\lambda} := \sum_x \mathbb{P}_p(o \leftrightarrow x) \Delta(o,x)^\lambda = \sum_{l\in \mathbb Z} (k-1)^{\lambda l} \mathbb{E}_p X_l^{-\infty,\infty} 
\]
so that $\chi_{p,0}$ is the usual susceptibility. As explained in \cite[Section 6.2]{Hutchcroftnonunimodularperc}, for each fixed $p\in [0,1]$ the tilted susceptibility is a convex function of $\lambda$ satisfying $\chi_{p,\lambda}=\chi_{p,1-\lambda}$, so that its minimum is always attained at $\lambda =1/2$.
We define the parameters $p_h$ and $p_t$, known as the \textbf{heaviness threshold} and \textbf{tiltability threshold}, by
\begin{multline*}
p_h=p_h(G,\Gamma) := \inf\bigl\{p: \text{there exist clusters of unbounded height}\bigr\}
\\= \inf\bigl\{p: X_0^{-\infty,\infty} = \infty \text{ with positive probability}\bigr\},
\end{multline*}
where the equivalence between these two definitions was proven by Timar \cite[Lemma 5.2]{timar2006percolation},
and 
\[
p_t =p_t(G,\Gamma):= \sup\bigl\{p: \chi_{p,1/2} < \infty\bigr\}=\sup\bigl\{p: \chi_{p,\lambda} < \infty \text{ for some $\lambda \in \mathbb R$}\bigr\}.
\]
Both the heaviness threshold and tiltability threshold are defined more generally for connected locally finite graphs equipped with a quasi-transitive nonunimodular group of automorphisms, and we refer the reader to \cite{Hutchcroftnonunimodularperc} for further discussion.
From the definitions, we obviously have $p_c \leq p_t \leq p_h \leq p_u$; the main result of \cite{Hutchcroftnonunimodularperc} is that $p_c<p_t$ for any connected, locally finite graph equipped with a quasitransitive nonunimodular group of automorphisms, so that $p_c<p_u$ for any graph admitting such a group of automorphisms. 
It was also shown in the proof of \cite[Theorem 2.9]{1804.10191} that $p_t(G,\Gamma) \leq p_{2\to 2}(G) \leq p_u(G)$ for any graph $G$ and any nonunimodular quasitransitive automorphism group $\Gamma$ of $G$.

Our final main result strengthens \Cref{thm:lamplighter_on_tree_p22} to establish equality between all four thresholds $p_h$, $p_t$, $p_{2\to 2}$, and $p_u$ for the examples we treat.
It suffices to prove that $p_t=p_u$ since the inequalities $p_t\leq p_h,p_{2\to 2}\leq p_u$ always hold.

\begin{theorem}
\label{thm:lamplighter_on_tree}
Let $T$ be a regular tree of degree $k\geq 3$. Suppose that $G$ is either $\LL(T)$ or $T\times H$ for some amenable Cayley graph $H$, and let 
 $\Gamma$ be the nonunimodular transitive group of automorphisms of $G$ that fix some given end of the underlying tree. Then  $p_u(G)=p_{2\to 2}(G)=p_t(G,\Gamma)=p_h(G,\Gamma)$.
\end{theorem}

The proof of this theorem is based on the \emph{Hammersley-Welsh argument} \cite{MR0139535} from the study of self-avoiding walk, and more specifically on the nonunimodular version of this argument developed in \cite{1709.10515}. Roughly speaking, we use this argument to show that the tilted susceptibility $\chi_{p,1/2}$ is finite whenever the Hausdorff dimension of the limit set is strictly less than $1/2$ via certain inequalities between generating functions. We also use \Cref{thm:lamplighter_on_tree} and other nonunimodular techniques from \cite{Hutchcroftnonunimodularperc} in our proof that the Hausdorff dimension of limit sets depends continuously on $p$ in the interval $(p_c,p_u]$.

\begin{remark}
The equality or strict inequality of the thresholds $p_t$, $p_h$, $p_{2\to 2}$, and $p_u$ are known to be sensitive to both the choice of graph and the choice of transitive automorphism group of the graph. Timar \cite{timar2006percolation} proved that if the subgraphs induced by the slabs $L_{n,m}$ for $n,m$ finite are all amenable (as is the case in our example) then $p_h=p_u$, a fact that can also be deduced from the relative Burton-Keane theorem of \cite{2409.12283}. It was conjectured by the first author in \cite{Hutchcroftnonunimodularperc} that the converse is also true, so that if slabs are nonamenable then $p_h<p_u$.
In \cite[Section 8.1]{Hutchcroftnonunimodularperc}, two different nonunimodular transitive groups of automorphisms of the $4$-regular tree are considered: The group $\Gamma$ fixing an end, and the group $\tilde \Gamma$ fixing a so-called $(1,1,2)$-orientation of the tree. For the first group $\Gamma$, one can compute that $p_t=p_{2\to 2}=1/\sqrt{3}<p_h=p_u=1$. For the second group $\tilde \Gamma$, it is proven that
\[
p_t=p_h=\frac{1}{6}+\frac{\sqrt{2}}{3}-\frac{1}{6}\sqrt{4\sqrt{2}-3},
\]
which is strictly less than both $p_{2\to 2}=1/\sqrt{3}$ and $p_u=1$. (Of course, $p_{2\to 2}$ and $p_u$ do not depend on the choice of transitive group of automorphisms.) The mechanism causing the equality $p_t=p_h$ to hold in this example is very similar to that causing the equality $p_t=p_u$ to hold for the lamplighter on the tree, but is easier to establish due to the simpler structure of the graph. See also \cite[Section 8.2]{Hutchcroftnonunimodularperc} for an example of a pair $(G,\Gamma)$ with $p_h<p_u<1$.
\end{remark}

\section{Proof}

\subsection{The Hausdorff dimension from supermultiplicative sequences}

For the remainder of the paper, we will fix a graph $G$ which is either $T\times H$ for an amenable transitive graph $H$ or $\LL(T)$, and let $\pi:G\to T$ be the associated projection map. We will work with the following characterization of the uniqueness threshold established in our companion paper \cite{2409.12283}.

\begin{lemma}
\label{lem:p_u_fiber_def}
Every cluster has finite intersection with every fiber $[x]$ a.s.\ if and only if $p\leq p_u$.
\end{lemma}

\begin{proof}[Proof of \cref{lem:p_u_fiber_def}]
Since the fiber $[o]$ is an amenable normal subgroup of a group for which $G$ is a Cayley graph, it follows from \cite[Theorem 1.9 and Proposition 1.9]{2409.12283} that the intersection $|K_o \cap [o]|$ is almost surely finite if and only if there is not a unique infinite cluster in $G$. The claim follows since the infinite cluster is not unique at $p_u$ by the results of \cite{2409.12283,MR1770624}.
\end{proof}

We fix an end of $T$ and define the height function $h$ and the levels $L_{n,m}$ as above. We write $\{x\xleftrightarrow[]{L_{n,m}} y\}$ for the event that $x$ is connected to $y$ by an open path in the slab $L_{n,m}$.
We also define $L_{-n}^*$ to be the set of vertices $x$ of $G$ for which $\pi(x)$ is a generation $n$ descendant of $\pi(o)$ (equivalently, the connected component of $o$ in $L_{-n,0}$) and let 
\[P_p(n) = \mathbb{P}_p(o \xleftrightarrow{L_{-n,0}} [x]) \quad \text{ and } \quad E_p(n) = \mathbb{E}_p[\#\{y\in [x] : o \xleftrightarrow{L_{-n,0}} y\}] \qquad \text{ where $x\in L_{-n}^*$};\]
both quantities are independent of the choice of $x\in L_{-n}^*$ by the symmetries of the graph.

\begin{lemma}
\label{lem:supermult}
The sequence $P_p(n)$ satisfies the supermultiplicative inequality 
\[P_p(n+m+1)\geq p P_p(n)P_p(m)\] for every $n,m \geq 1$.
\end{lemma}

\begin{proof}[Proof of Lemma~\ref{lem:supermult}]
Fix $x \in L_{-n}^*=L_{-n}^*(o)$ and $y\in L_{-m-1}^*(x) \subseteq L_{-n-m-1}^*(o)$. Condition on the configuration in $L_{-n,0}$ and suppose $0$ is connected to $[x]$. For each vertex $z$ in $[x]$, we can connect to $[y]$ inside the slab $L_{-m-1,0}(x)=L_{-n-m-1,-n}(o)$, without using any edge whose status has already been revealed, by first taking a tree edge down one level and then connecting from that other endpoint to $[y]$ within the slab $L_{-m-1,-1}(z)$. This event has conditional probability at least $pP_p(m)$ by the Harris-FKG inequality.
\end{proof}

Using Fekete's lemma, Lemma~\ref{lem:supermult} implies that the limit
\[
\beta_p^* := -\lim_{n\to\infty} \frac{\log P_p(n)}{n \log (k-1)} = - \sup_n\frac{\log p P_p(n)}{n \log (k-1)} 
\]
is well-defined, so that $P_p(n)=(k-1)^{-\beta_p^*n \pm o(n)}$ as $n\to\infty$. This quantity has the following basic properties:

\begin{lemma}
\label{lem:beta*_strictly_decreasing}
 $\beta_p^*$ is left-continuous in $p$ and is a strictly decreasing function of $p$ on the set $\{p\in [0,1]:\beta_p^*>0\}$.
\end{lemma}

\begin{proof}[Proof of Lemma~\ref{lem:beta*_strictly_decreasing}]
The expression for $\beta_p^*$ as the infimum of the continuous functions 
 ensures that it is upper semicontinuous, and since it is decreasing it must be left-continuous as claimed. We now prove that $\beta_p^*$ is strictly decreasing when it is positive.
It is a general fact \cite[Theorem 2.38]{grimmett2010percolation} that if $A$ is an increasing event then
$\mathbb{P}_{p^\theta}(A) \geq \mathbb{P}(A)^\theta$
for every $p,\theta \in [0,1]$. Applying this inequality to the event $\{o\xleftrightarrow{L_{-n,0}} [x]\}$ with $x\in L_{-n}^*$ yields that
$P_{p^\theta}(n) \geq P_p(n)^\theta$
for every $p,\theta \in [0,1]$ and $n\geq 0$, and hence that
$\beta_{p^\theta}^* \leq \theta \beta_p^*$
for $p,\theta \in [0,1]$. 
\end{proof}

We next argue that the Hausdorff dimension of limit sets can be expressed in terms of $\beta_p^*$, meaning that for the rest of the paper we can work with $\beta_p^*$ rather than directly with the Hausdorff dimension.

\begin{lemma}
\label{lem:beta_dim}
Let $\Lambda$ be the set of limit points in $\partial T$ of the projected cluster $\pi(K_o)$. On the event that $K_o$ is infinite, its Hausdorff dimension $\delta_H(\Lambda)$ is $\mathbb{P}_p$-almost surely equal to $1-\beta_p^*$ for each $p_c<p_c\leq p_u$ and is almost surely equal to $1$ for $p>p_u$.
\end{lemma}

\begin{proof}[Proof of \Cref{lem:beta_dim}]
The fact that all infinite clusters have accumulation sets in $\partial T$ with the same, non-random Hausdorff dimension is an immediate consequence of the indistinguishability theorem of Lyons and Schramm \cite{LS99}. When $p>p_u$ this accumulation set is the entire boundary, so that its dimension is trivially equal to $1$. We need to prove that this dimension is equal to $1-\beta_p^*$ for $p_c<p\leq p_u$.

We begin by proving the upper bound on the Hausdorff dimension.
 It follows from \Cref{lem:p_u_fiber_def} and an easy finite-energy argument that if $p_c<p\leq p_u$ then the origin $o$ belongs, with positive probability, to an infinite cluster in which every other vertex belongs to a level of negative height. Since the Hausdorff dimension of the limit set is non-random, it therefore suffices to consider the dimension of the limit set when we restrict percolation to the subgraph whose corresponding tree vertices are descendants of $\pi(o)$. In this graph, the set $\{ x \in \pi(L_{-n}^*) : o \xleftrightarrow{L_{-n,0}} x\}$ encodes an open cover $\mathcal{C}_n$ of the limit set $\Lambda$, namely the set of ends of the tree corresponding to the subtree lying below each $x$ in the set. Since each of these sets has diameter $(k-1)^{-n}$ in $\partial T$, we have for each $\theta>0$ that
\[
\mathbb{E}_p \left[\sum_{U\in \mathcal{C}_n} \operatorname{diam}(U)^\theta\right] = (k-1)^n P_p(n) (k-1)^{-\theta n} = (k-1)^{(1-\theta-\beta_p^*)n\pm o(n)}
\]
as $n\to \infty$. Thus, it follows by Markov's inequality that the Hausdorff dimension of the limit set is at most $1-\beta_p^*$ almost surely. 

The matching lower bound follows by a very similar argument to that used in \cite{HL00,lalley1998limit}, in which one uses the level structure to construct embedded Galton-Watson processes whose limit-set dimension can be computed using the Hawkes-Lyons theorem \cite{Hawkes81,lyons1990random}. Similar constructions are carried out carefully in the proof of \Cref{lm:beta for direct product_flat}; we omit further details here.
\end{proof}

\subsection{Probabilities vs Expectations}

Our next goal is to prove that the two quantities $P_p(n)$ and $E_p(n)$ coincide up to subexponential factors when $p<p_u$. 

\begin{lemma}
\label{lem:probabilities_and_expectations}
For each $0<p<p_u$, there exist constants $C_p,\tilde C_p < \infty$ such that
\[
P_p(n) \leq E_p(n) \leq C_p P_p(n) \log \frac{e}{P_p(n)} \leq \tilde C_p (n+1) P_p(n) 
\]
for every $n\geq 0$. In particular, if $p<p_u$ then $E_p(n)=(d-1)^{-\beta_p^*n\pm o(n)}$ as $n\to\infty$.
\end{lemma}

The factor of $e$ inside the log ensures that $\log \frac{e}{P_p(n)} \geq 1$. We will deduce \Cref{lem:probabilities_and_expectations} from the following lemma, which is an immediate consequence of the main results of our companion paper \cite{2409.12283}.

\begin{lemma}\label{lm:decay of intersection with subgroup}
If $p<p_u$ then $\sup_x \mathbb{P}_p(|K_o\cap [x]|\geq n)$ decays exponentially fast in $n$. In particular, $\sup_x \mathbb{E}_p|K_o\cap [x]|<\infty$ for every $p<p_u$. 
\end{lemma}

\begin{proof}[Proof of \Cref{lm:decay of intersection with subgroup}]
 In the notation of \cite{2409.12283}, \Cref{lem:p_u_fiber_def} implies that $p_u(G)=p_c([o];G)$. (This equality holds with $[o]$ replaced by any amenable normal subgroup of a group and $G$ a Cayley graph of that group.) The claimed exponential decay therefore follows by relative sharpness of the phase transition \cite[Theorem 1.8]{2409.12283} applied to the subgroup $[o]$ together with \cite[Lemma 5.5]{2409.12283}, which allows us to bound $\mathbb{P}_p(|K_o\cap [x]|\geq n) \leq 2 \mathbb{P}_p(|K_o\cap [o]|\geq n)$.
\end{proof}

\begin{proof}[Proof of Lemma~\ref{lem:probabilities_and_expectations}]
The first inequality follows from Markov's inequality; we focus on the second.
If $p<p_u$, we are by Lemma \ref{lm:decay of intersection with subgroup} that there exists $c_p>0$ such that
\[
\mathbb{P}_p(|K_o\cap [x]| \geq n) \leq e^{-c_p n}
\]
for every vertex $x\in G$ and every $n\geq 0$. 
Equivalently, there exists a constant $C_p<\infty$ such that the moment estimate
\[
\mathbb{E}_p\left[|K_o\cap [x]|^q \right] \leq \int_0^\infty t^{q-1} e^{-c_p t} \mathrm{d} t \leq (C_p q)^q
\]
holds for every vertex $x$ and every $q\in [1,\infty)$. Thus, we have by H\"older's inequality that
\begin{multline*}
\mathbb{E}_p\left[\#\{y\in [x] : o \xleftrightarrow{L_{-n,0}} y\} \right] \leq \mathbb{P}_p(o \xleftrightarrow{L_{-n,0}} [x])^{(q-1)/q} \mathbb{E}_p\left[\#\{y\in [x] : o \xleftrightarrow{L_{-n,0}} y\}^{q} \right]^{1/q}
\\\leq 
\mathbb{P}_p(o \xleftrightarrow{L_{-n,0}} [x])^{(q-1)/q} C_p q
\end{multline*}
for every $n \geq 0$, $x \in L_{-n}$, and $q\geq 1$. Taking $q=\max\{1,-\log P_p(n)\}$ yields that for every $p<p_u$ there exists a constant $\tilde C_p$ such that
\[
E_p(n) \leq \tilde C_p P_p(n) \log \frac{e}{P_p(n)}
\]
for every $n\geq 0$ as claimed. The second inequality follows since $P_p(n)\geq p^n$.
\end{proof}

\subsection{Backscattering}\label{sec:beta>1/2}

In this section we prove that if $\beta^*_p<\frac{1}{2}$ then $p\geq p_u$. We  treat the lamplighter $\LL(T)$ and the product $T\times H$ separately, beginning with the easier case of the lamplighter. In both cases, the idea is to show that if $\beta^*_p<1/2$ then it is easy for a cluster to have a large intersection with the fiber $[o]$ by first growing to a large negative level $L_{-n}$ and then going back up the tree towards height zero, a fact referred to as the ``Backscattering Principle'' by Lalley and Selke \cite{lalley1997hyperbolic} in the context of branching random walk.

\begin{lemma}[Lamplighter backscattering]
\label{lem:coming_back_up_with_lamps} If $G=\LL(T)$ then the inequality\label{lm:beta(LL(T))}
\begin{equation}
\label{eq:LL_backscattering}
\mathbb{E}_p |K_o \cap [o]| \geq p^3 (k-1)^{n+1} P_p(n) E_p(n)
\end{equation}
holds for every $p\in [0,1]$ and $n\geq 0$. As a corollary, $\beta^*_p\geq 1/2$ for every $p<p_u$.
\end{lemma}

\begin{proof}[Proof of Lemma~\ref{lem:coming_back_up_with_lamps}]
We can form open paths from $o$ to $[o]$ by first connecting to $L_{-n}^*$ inside the slab $L_{-n,0}$, taking one step down the tree, flipping the lamp at that vertex, taking one step back up, then connecting back up to $[o]$ within $L_{-n,0}$. The paths formed by doing this can only end at the same vertex if they started by connecting to the same fibre in $\pi(L_{-n}^*)$, since otherwise they cannot have the same lamp configuration at their endpoint. Thus, the claimed inequality \eqref{eq:LL_backscattering} follows by the Harris-FKG inequality.
If $\beta_p^*<1/2$ then the right hand side of \eqref{eq:LL_backscattering} converges to $\infty$ as $n\to \infty$, so that if $\beta_p^*<1/2$ then $\mathbb{E}_p \left[|K_o \cap [o]|\right]=\infty$. It follows from \Cref{lm:decay of intersection with subgroup} that $\beta_p^*\geq 1/2$ for every $p<p_u$ as claimed.
\end{proof}

We now turn to the product $T\times H$, for which our proof is very different. For each $n\geq 1$, we let $x\in L_{-n}^*$ and define $\mathscr{A}(x)$ to be the event that $o$ is connected to $x$ by an open path in $L_{-n,0}$ that visits $L_{-n}^*$ only at its final point. 
By symmetry, if we let $(v,h)$ be a vertex of $T\times H$ expressed in coordinates, the probability of $\mathscr{A}((v,h))$ does not depend on the choice of $v\in \pi(L_{-n}^*)$, and we denote by $Q_p(n)$ the common probability of the events $\mathscr{A}((v,o_H))$ for $v\in \pi(L_{-n}^*)$ and $o_H$ the root of $H$ (chosen so that $o=(o_T,o_H)$). We have trivially that $Q_p(n)\leq P_p(n)$; our next goal is to prove that these two quantities have the same exponential rate of decay.

\begin{lemma}\label{lm:beta for direct product_flat}
Suppose that $G=T\times H$ with $H$ an amenable transitive graph. Then
\begin{equation}
\label{eq:flat_connections}
Q_p(n)
 = (k-1)^{-\beta_p^* n + o(n)}
\end{equation}
as $n\to \infty$ for each $p\in (0,1]$, where $o_H$ denotes the identity in $H$.
\end{lemma}

(Note that while $Q_p(n)$ is supermultiplicative and therefore has a well-defined rate of exponential decay by Fekete's lemma, the proof will not make use of this fact.)

\begin{proof}[Proof of \Cref{lm:beta for direct product_flat}]
It suffices to prove the lower bound, with the upper bound following immediately from the definitions.
Fix $p\in (0,1]$ and $\beta_1 > \beta_p^*$. If $N$ is sufficiently large then
\begin{equation*}
P_{p}(N-1)>\frac{1}{p} (k-1)^{-\beta_1 (N-1)}.
\end{equation*}
Fix one such $N$, and note that if $x \in L_{-N}^*$ and $\mathscr{B}(x)$ denotes the event that $o$ is connected to $[x]$ via an open path in $L_{-N,0}$ that intersects $L_{-N}^*$ only at its last vertex, then 
\[
\mathbb{P}_p(\mathscr{B}(x)) \geq p P_{p}(N-1) > (k-1)^{-\beta_1 (N-1)}.
\]
  Given two vertices $x,y$ with $h(x,y)=-N$, we also define $\mathscr{A}(x,y)$ to be the event that there exists an open path from $x$ to $y$ in $L_{-N,0}(x)$ that visits $L_{-N}^*$ only at its final point.

We will now use the events $\mathscr{A}$ to couple percolation on $G$ with a certain branching random walk on $H$. 
Let $Z_0=\{o\}$ and define $Z_1$ as follows: For each vertex $v$ of $T$ such that $\mathscr{B}(x)$ holds for some (and hence every) $x\in \pi^{-1}(v)$, pick an element from the set $\{x\in \pi^{-1}(v): \mathscr{A}(x)$ holds$\}$ uniformly at random and include this element in $Z_1$. Thus, $Z_1$ contains exactly one point from each fiber $N$ levels below $o$ that is connected to $o$ in the manner required by the events $\mathscr{A}$. Inductively, we define $Z_k$ for $k\geq 2$ by, for each element $x\in Z_{k-1}$ and every vertex $v$ of $T$ such that  $\mathscr{A}(x,y)$ holds for some $y\in \pi^{-1}(v)$, picking an element from the set $\{y\in \pi^{-1}(v): \mathscr{A}(x,y)$ holds$\}$ uniformly at random and including this element in $Z_k$. The sequence of sets $Z=(Z_i)_{i\geq 0}$ is defined so that $(|Z_i|)_{i\geq 0}$ is a branching process; in particular, given some $x\in Z_i$, the part of the process $Z$ that describes the descendants of $x$ is conditionally independent given $(Z_0,\ldots,Z_i)$ from the parts of the process that describe the descendants of the other points in $Z_i$. (Indeed, these different parts of the process depend on the percolation configurations in disjoint subgraphs of $T\times H$.)

Next, observe that if $\pi_H$ denotes the projection $T\times H \to H$ then $(\pi_H(Z_i))_{i\geq 0}$ evolves as a branching random walk on $H$. (In this branching random walk the steps taken by the offspring of a particle are not necessarily independent of the total number of offspring of a particle, since both depend on the percolation configuration restricted to an appropriate slab. This does not cause any problems.) Letting $P$ be the transition kernel on $H$ defined by
\[P(h)=\frac{\mathbb{E} |Z_1 \cap \pi_H^{-1}(h)|}{\mathbb{E}|Z_1|},\]
and letting $P(h_1,h_2)=P(h_1^{-1}h_2)$ denote the associated transition matrix, we have by induction that
\[
\mathbb{E} |Z_r \cap \pi^{-1}(h)| = \sum_{g\in H}\mathbb{E}|Z_{r-1}\cap \pi^{-1}(g)| \mathbb{E}|Z_{1}\cap \pi^{-1}(g^{-1}h)| = (\mathbb{E}|Z_1|)^r P^r(o_H,h)
\]
for each $h\in H$ and $r\geq 0$. Since $H$ is amenable, it follows by Kesten's theorem \cite{kesten1959full} that $P^{2r}(o_H,o_H)=e^{-o(r)}$ as $r\to \infty$, so that
\[
\mathbb{E} |Z_{2r} \cap \pi^{-1}(o_H)| = (\mathbb{E}|Z_1|)^{2r+o(r)}
\]
as $r\to \infty$. Since $\mathbb{E}|Z_1| = (k-1)^N\mathbb{P}_p(\mathscr{B}(x)) > (k-1)^{(1-\beta_1) N}$, where $x\in L_{-N}^*$, it follows that if we fix $\beta_1 < \beta_2$ then 
\[
\mathbb{E} |Z_{2r} \cap \pi^{-1}(o_H)| \geq (k-1)^{(1-\beta_2)2rN}
\]
for all sufficiently large $r$. Fix one such $r$. We can define another branching process by setting $W_0=\{o\}$, $W_1=Z_{2r}\cap \pi^{-1}(o_H)$ and inductively setting $W_\ell$ to be the set of points in $Z_{2\ell r}$ that belong to $\pi^{-1}(o_H)$ and whose ancestor in $Z_{2(\ell-1)r}$ belonged to $W_{\ell-1}$. Since $(|W_\ell|)_{\ell\geq 0}$ is a branching process, we have that
\[
(k-1)^{2\ell r N} Q_p(2\ell rN) \geq \mathbb{E}|W_\ell| = (\mathbb{E}|W_1|)^\ell \geq (k-1)^{(1-\beta_2) 2\ell r N}
\]
for every $\ell \geq 0$. A lower bound of the same form follows for all $Q_p(n)$ since $Q_p(n+1)\geq pQ_p(n)$ for every $n\geq 1$, and the claimed exponential decay rate of $Q_p(n)$ follows since $\beta_2>\beta_1>\beta_p^*$ were chosen arbitrarily.
\end{proof}

\begin{lemma}[Backscattering for products with trees]\label{lm:beta for direct product}
 If $G=T\times H$ with $H$ an amenable transitive graph then 
 $\beta
_{p}^{\ast }\geq1/2$ for every $p<p_u$. 
\end{lemma}

\begin{proof}[Proof of \Cref{lm:beta for direct product}]
We will assume that $\beta_p^*<1/2$ and prove that $p\geq p_u$.
Let $o_T$ be the root of $T$, let $h$ and $g$ be two vertices of $H$, and let $\beta_p^*<\beta<1/2$ be fixed. By \Cref{lm:beta for direct product_flat} and the Harris-FKG inequality, we may fix $r\geq 1$ sufficiently large that
\begin{multline*}
\mathbb{P}_p(\mathscr{A}((o_T,h),(v,h)) \cap \mathscr{A}((o_T,g),(v,g))
\geq \mathbb{P}_p(\mathscr{A}((o_T,h),(v,h)))\mathbb{P}_p(\mathscr{A}((o_T,g),(v,g))
\\\geq (k-1)^{-2\beta r} > (k-1)^{-r}
\end{multline*}
for every $v\in \pi(L_{-r}^*)$. Note that this constant $r$ can be taken independently of the choice of $h,g$.
 We now define yet another embedded branching process $Y^{h,g}=(Y^{h,g}_i)_{i\geq 0}$ by setting $Y^{h,g}_0=\{o_T\}$ and for each $i\geq 0$ letting $Y^{h,g}_{i+1}$ be the set of vertices $v$ in $\pi(L_{-(i+1)r}^*)$ for which there exists a point $u \in Y^{h,g}_i$ such that the event
$\mathscr{A}((u,h),(v,h)) \cap \mathscr{A}((u,g),(v,g))$ holds. This sequence has the property that $(|Y^{h,g}_i|)_{i\geq 0}$ is a branching process in the usual sense. Since 
\[
\mathbb{E}|Y^{h,g}_1| = (k-1)^r\mathbb{P}_p(\mathscr{A}((o_T,h),(v,h)) \cap \mathscr{A}((o_T,g),(v,g)) > 1,
\]
this branching process is supercritical, so that it survives forever with positive probability. Moreover, since $|Y_1^{h,g}|$ is bounded above by the constant $(k-1)^{r}$, the probability that $Y^{h,g}$ survives forever is bounded away from zero uniformly in the choice of $h$ and $g$; this follows from the characterisation of the extinction probability as the smallest positive fixed point of the probability generating function \cite[Proposition 5.4]{LP:book} together with the fact that the second derivative of the probability function can be bounded in terms of the supremum of the support of the relevant random variable.

Now observe that if $Y^{h,g}$ survives forever then the clusters of $(o_T,h)$ and $(o_T,g)$ must either be equal or come within distance $d(h,g)$ of each other infinitely often. In either case, the two clusters will become a.s.\ equal when we increase $p$ to any $p'>p$ in the standard monotone coupling of percolation at different values of $p$, so that
\[
\inf_{h,g} \mathbb{P}_{p'}((o_T,h) \leftrightarrow (o_T,g)) \geq \inf_{h,g} \mathbb{P}_p(Y^{h,g} \text{ survives forever})>0
\]
for any $p'>p$. It follows by Fatou's lemma that there is a cluster having infinite intersection with $[o]$ with positive probability for every $p'>p$ and hence by \Cref{lem:p_u_fiber_def} that
 $p\geq p_u$ as claimed.
\end{proof}

Together, \Cref{lem:beta_dim,lem:coming_back_up_with_lamps,lm:beta for direct product} establish all of \Cref{thm:dimension_jump} except for the claim that the dimension depends continuously on $p$ below $p_u$. This last part of the theorem will be established at the end of the paper after we prove that $p_u=p_t=p_{2\to 2}=p_h$ in the next section.

\subsection{The Hammersley-Welsh argument}
In this section we prove that $p_u=p_t$ and hence that $p_u=p_t=p_{2\to2}=p_h$, completing the proof of \Cref{thm:lamplighter_on_tree_p22,thm:lamplighter_on_tree}. More concretely, our goal is to prove that the tilted susceptibility $\chi_{p,1/2}$ is finite whenever $\beta_p^*>1/2$, which will yield the claim in conjunction with \Cref{lem:coming_back_up_with_lamps,lm:beta for direct product}. As mentioned in the introduction, the proof will be based on the Hammersley-Welsh argument \cite{MR0139535} usually used to relate generating functions for self-avoiding walks and self-avoiding bridges, and in particular
the nonunimodular version of the Hammersley-Welsh argument developed in \cite{1709.10515}.

We begin by introducing some more notation. We define
\[
D_p(n) := \mathbf{E}_p X_{-n}^{-n,0} = (k-1)^n E_p(n),
\]
where the relation with $E_p(n)$ follows since $L_{-n}^*$ is equal to the set of vertices $x\in L_{-n}$ such that $[x]$ is connected to the origin in by a (not necessarily open) path in $L_{-n,0}$, and define
\[
U_p(n) := \mathbf{E}_p X_n^{0,n} = (k-1)^{-n} D_p(n) = E_p(n),
\]
where the second equality follows from the tilted mass-transport principle. (Although the quantities $U_p(n)$ and $E_p(n)$ are equal, we prefer to give them different names corresponding to their two different meanings in the model.)
The following key lemma, which encapsulates the Hammersley-Welsh method, relates the slab quantities $D_p(n)$ to the global quantity $\chi_{p,1/2}$.

\begin{lemma}
\label{lem:Hammersley-Welsh}
The tilted susceptibility satisfies
\[\chi_{p,1/2} \leq (\mathbb{E}_p|K_o\cap [o]|)^2 \exp\left[ 2\sum_{n=0}^\infty (k-1)^{-n/2} D_p(n) \right]\]
for every $p\in [0,1]$.
\end{lemma}

Before proving this lemma, let us see how it implies \Cref{thm:lamplighter_on_tree}.

\begin{proof}[Proof of \Cref{thm:lamplighter_on_tree}]
Since the inequalities $p_t\leq p_h,p_{2\to 2}\leq p_u$ always hold, it suffices to prove that $p_u \leq p_t$, or equivalently that $\chi_{p,1/2}<\infty$ for every $p<p_u$.
 Fix $p<p_u$. It follows from Lemma \ref{lm:decay of intersection with subgroup} that $\mathbb{E}_p|K_o\cap [o]|<\infty$, while \Cref{lem:coming_back_up_with_lamps,lm:beta for direct product} imply that $\beta_p^*>1/2$. Since $D_p(n)=(k-1)^n E_p(n)= (k-1)^{(1-\beta_p^*) n \pm o(n)}$ by Lemma~\ref{lem:probabilities_and_expectations}, it follows immediately from Lemma~\ref{lem:Hammersley-Welsh} that  $\chi_{p,1/2}<\infty$ as claimed.
\end{proof}

In order to prove Lemma~\ref{lem:Hammersley-Welsh}, we first prove the following lemma relating the tilted susceptibility to a half-space version of itself.

\begin{lemma}
\label{lem:half_space_to_full_space}
The generating function defined by
\[
\mathcal{H}_{p,\lambda} :=  \sum_{n=0}^\infty (k-1)^{-\lambda n} \mathbb{E}\left[ X_{-n}^{-\infty,0}\right]
\]
satisfies the inequality
$\chi_{p,\lambda} \leq \mathcal{H}_{p,\lambda} \mathcal{H}_{p,1-\lambda}$
for every $p\in [0,1]$ and $\lambda \in \mathbb{R}$.
\end{lemma}

\begin{proof}[Proof of Lemma~\ref{lem:half_space_to_full_space}] Taking a union bound over the maximal height reached by an open path and using the BK inequality yields that
\[
\mathbb{E}_p X_{n}^{-\infty,\infty} \leq \sum_{l=n\vee 0}^\infty \mathbb{E}_p X_{l}^{-\infty,l} \mathbb{E}_p X_{n-l}^{-\infty,0} = \sum_{l=n\vee 0}^\infty (d-1)^{-l} \mathbb{E}_p X_{-l}^{-\infty,0} \mathbb{E}_p X_{n-l}^{-\infty,0}
\]
for every $n \in \mathbb{Z}$, where the equality follows from the tilted mass-transport principle. Multiplying by $(k-1)^{-\lambda n}$ and summing over $n$, we obtain that
\begin{equation*}
\chi_{p,\lambda} \leq \sum_{n = - \infty}^\infty \sum_{l=n\vee 0}^\infty (k-1)^{-(1-\lambda) l} \mathbb{E}_p X_{-l}^{-\infty,0} (k-1)^{-\lambda(n-l)}\mathbb{E}_p X_{n-l}^{-\infty,0} 
= \mathcal{H}_{p,\lambda} \mathcal{H}_{p,1-\lambda}
\end{equation*}
as claimed.
\end{proof}

\begin{proof}[Proof of Lemma~\ref{lem:Hammersley-Welsh}]
It suffices by Lemma~\ref{lem:half_space_to_full_space} to prove that
\begin{equation}
\label{eq:Hammersley-Welsh_claim}
\mathcal{H}_{p,1/2} \leq \mathbb{E}_p\left[|K_o\cap [o]|\right] \exp\left[ \sum_{n=0}^\infty (k-1)^{-n/2} D_p(n) \right].\end{equation}
Suppose that $o$ is connected to $x$ by some open path $\gamma$ inside $L_{-\infty,0}$. Let 
\[\sigma_0 = \max\{0\leq i \leq \operatorname{len}(\gamma) : h(\gamma(i)) = 0\}\]
be the last time the path $\gamma$ visits a vertex of height $0$, and recursively let
\[
\sigma_j = \begin{cases} \max\{\sigma_{j-1} \leq i \leq \operatorname{len}(\gamma) : h(\gamma(i)) = \min_{j \geq \sigma_{j-1}} h(\gamma(j))\} & j \text{ odd}\\
 \max\{\sigma_{j-1} \leq i \leq \operatorname{len}(\gamma) : h(\gamma(i)) = \max_{j\geq \sigma_{j-1}} h(\gamma(j))\} & j \text{ even}
\end{cases},
\]
stopping when $\sigma_j=\operatorname{len}(\gamma)$. Thus, for each $j$, the part of $\gamma$ between $\sigma_j$ and $\sigma_{j+1}$ crosses a slab from top to bottom if $j$ is even and from bottom to top if $j$ is odd, and in particular is contained inside this slab of heights between its first and last points. Moreover, the absolute value of the height difference between $\gamma(\sigma_j)$ and $\gamma(\sigma_{j+1})$ is strictly decreasing in $j$.

Using the BK inequality and a union bound, we obtain that
\[
\mathbb{E} X_{-n}^{-\infty,0} \leq \mathbb{E}_p |K_o \cap [o]| \mathbf{1}(n=0) + \mathbb{E}_p |K_o \cap [o]| \sum_k \sum_{\mathbf{s} \in S_{n,k}} \prod_{i=0}^{\lfloor (k-1)/2\rfloor} D_p(s_{2i+1}) \prod_{i=0}^{\lfloor k/2 \rfloor} U_p(s_{2i}),
\]
where $S_{l,n}$ is the set of strictly decreasing sequences of positive integers $\mathbf{s}=(s_1,\ldots,s_l)$ such that $\sum_{i=1}^l (-1)^{l+1} s_i=n$. (See \cite[Section 5.4]{Hutchcroftnonunimodularperc} for more detailed proofs of similar ``up and down'' inequalities.)
Using the fact that $D_p(n)=(k-1)^{n/2} \sqrt{D_p(n)U_p(n)}$ and $U_p(n)=(k-1)^{-n/2} \sqrt{D_p(n)U_p(n)}$, we have that
\[
\prod_{i=0}^{\lfloor (l-1)/2\rfloor} D_p(s_{2i+1}) \prod_{i=0}^{\lfloor l/2 \rfloor} U_p(s_{2i}) = (k-1)^{n/2} \prod_{i=1}^l \sqrt{D_p(s_i)U_p(s_i)}
\]
for every $\mathbf{s}\in S_{l,n}$, and hence that
\[
\frac{(k-1)^{-n/2}\mathbb{E} X_{-n}^{-\infty,0}}{\mathbb{E}_p |K_o \cap [o]|} \leq \mathbf{1}(n=0)+ \sum_l \sum_{\mathbf{s}\in S_{n,k}} \prod_{i=1}^l \sqrt{D_p(s_i)U_p(s_i)}.
\]
Letting $S_l$ denote the set of strictly decreasing sequences of positive integers $\mathbf{s}=(s_1,s_2,\ldots,s_l)$ and noting that every such sequence is in $S_{n,l}$ for exactly one $n\geq 1$, it follows that
\[
\sum_{n=0}^\infty \frac{(k-1)^{-n/2}\mathbb{E} X_{-n}^{-\infty,0}}{\mathbb{E}_p |K_o \cap [o]|} \leq  1+\sum_{l=0}^\infty \sum_{\mathbf{s}\in S_{l}} \prod_{i=1}^l \sqrt{D_p(s_i)U_p(s_i)}.
\]
Now, observe that for any non-negative function $f:\{1,2,\ldots\}\to [0,\infty]$, we have that
\[
1+\sum_{l=0}^\infty \sum_{\mathbf{s}\in S_{l}} \prod_{i=1}^l f(s_i) = \prod_{i=1}^\infty (1+f(i)) \leq \exp\left[\sum_{i=1}^\infty f(i) \right],
\]
where the inequality on the right hand side follows from the elementary inequality $1+x\leq e^x$.
Applying this inequality with $f(i)=\sqrt{D_p(i)U_p(i)}$, it follows that 
\[
\sum_{n=0}^\infty(k-1)^{-n/2}\mathbb{E} X_{-n}^{-\infty,0} \leq \mathbb{E}_p |K_o \cap [o]| \exp\left[\sum_{i=1}^\infty \sqrt{D_p(i)U_p(i)} \right],
\]
which is equivalent to the claimed inequality \eqref{eq:Hammersley-Welsh_claim} since $D_p(i)=(k-1)^{i/2}\sqrt{D_p(i)U_p(i)}$.
\end{proof}

We are now ready to conclude the proof of \Cref{thm:dimension_jump}.

\begin{proof}[Proof of \Cref{thm:dimension_jump}] 
\Cref{lem:beta_dim,lem:coming_back_up_with_lamps,lm:beta for direct product}  already establish every claim in the theorem other than the fact that $\beta_p^*$ is continuous on $(0,p_u]$. By \Cref{lem:beta*_strictly_decreasing}, it suffices to prove that $\beta_p^*$ is right-continuous on $(0,p_u)$. This claim follows easily from the equality $p_t=p_u$ together with the results of \cite{Hutchcroftnonunimodularperc} as we now explain.
In \cite[Section 5.3]{Hutchcroftnonunimodularperc}, it is shown via supermultiplicativity considerations that the quantity
\[
\alpha_p := -\lim_{n\to\infty}\frac{ \log \mathbb{P}_p (X_n^{0,n} > 0)}{n \log (k-1)} 
\]
is well-defined and left-continuous for all $0<p \leq 1$. (Note that the simple structure of the modular function and the level sets in our example lets us express this quantity using simpler notation than is required in the general nonunimodular transitive setting.)
Moreover, it is proven in \cite[Sections 5.5 and 5.6]{Hutchcroftnonunimodularperc} that if $p<p_t$ then the limit
\[
\beta_p :=  -\lim_{n\to\infty}\frac{ \log \mathbb{E}_p \left[X_n^{-\infty,n}\right]}{n \log (k-1)}
\]
is well defined and equal to $\alpha_p$. 
 Since $\beta_p$ is defined using a \emph{submultiplicative} sequence, it is \emph{right} continuous, and the equality of $\alpha_p$ and $\beta_p$ on $(0,p_t)$ ensures that they are continuous on the same interval. As such, it suffices to prove that $\beta_p^*=\beta_p$ for all $0<p<p_t=p_u$.
This in turn follows easily from the inequalities
\[
\mathbb{E}_p X_n^{0,n} \leq \mathbb{E}_p X_n^{-\infty,n} \leq \mathbb{E}_p X_n^{0,n}  \mathbb{E}_p X_0^{\infty,\infty},
\]
the first of which is trivial and the second of which is a simple consequence of the BK inequality (see \cite[Section 5.4]{Hutchcroftnonunimodularperc}), together with the fact that $\mathbb{E}_p X_0^{\infty,\infty}\leq \chi_{p,1/2}$ is finite for $p<p_t$.
\end{proof}

\subsection*{Acknowledgements}
This work was supported by NSF grant DMS-1928930 and a Packard Fellowship for Science and Engineering.

\footnotesize{

\bibliographystyle{abbrv}
\bibliography{big_bib_file}
}
\end{document}